\documentclass[3p,10pt]{elsarticle}

\usepackage[utf8x]{inputenc}
\usepackage{amsmath,amsthm}
\usepackage{amssymb}
\usepackage[all]{xy}

\newtheorem{theorem}{Theorem}[section]
\newtheorem{auxtheorem}[theorem]{Addendum}
\newtheorem{lemma}[theorem]{Lemma}
\newtheorem{prop}[theorem]{Proposition}
\newtheorem{cor}[theorem]{Corollary}

\newtheorem*{theorem*}{Theorem}
\newtheorem*{EMC}{EMC}

\theoremstyle{definition}

\newtheorem*{defn*}{Definition}

\def\Xint#1{\mathchoice
	{\XXint\displaystyle\textstyle{#1}}%
	{\XXint\textstyle\scriptstyle{#1}}%
	{\XXint\scriptstyle\scriptscriptstyle{#1}}%
	{\XXint\scriptscriptstyle\scriptscriptstyle{#1}}%
    \!\int}
\def\XXint#1#2#3{{\setbox0=\hbox{$#1{#2#3}{\int}$}
     \vcenter{\hbox{$#2#3$}}\kern-.5\wd0}}

\newcommand{\Aut}[1]{\textup{Aut}(#1)}

\newcommand{\CComplex}{\mathbf{C}}
\newcommand{\Cl}{\mathcal{C}\negthinspace\ell}
\newcommand{\condabsc}[2]{\alpha_{#1}(\cyclic{#1}^{#2})}
\newcommand{\condc}[1]{\gamma(#1)}
\newcommand{\condb}[1]{\beta(#1)}
\newcommand{\conda}[1]{\alpha_p(#1)}
\newcommand{\cyclic}[1]{C_{#1}}

\newcommand{\fdindex}[2]{[#1:#2]}
\newcommand{\ffield}[1]{\mathbb{F}_{#1}}

\newcommand{\Gal}[2]{\textup{Gal}(#1/#2)}

\newcommand{\grindex}[2]{\left(#1 : #2\right)}
\newcommand{\grossO}[1]{\textup{O}\left(#1\right)}
\newcommand{\grossOmega}[1]{\Omega\left(#1\right)}
\newcommand{\grossTheta}[1]{\Theta\left(#1\right)}

\newcommand{\iso}{\simeq}

\newcommand{\kleino}[1]{\textup{o}\left(#1\right)}

\newcommand{\local}[2]{#1_{\mf{#2}}}
\newcommand{\mc}[1]{\mathcal{#1}}
\newcommand{\mf}[1]{\mathfrak{#1}}
\newcommand{\N}[1]{\Norm{\mathfrak{#1}}}
\newcommand{\Norm}[1]{\|#1\|}

\newcommand{\open}[2]{1+\mf{#1}^{#2}}
\newcommand{\Open}[2]{\langle\open{#1}{#2}\rangle}
\newcommand{\operp}{\bigcirc\kern-1.17em\perp}
\newcommand{\val}[1]{v_{\mf{#1}}}
\newcommand{\ord}[2]{\val{#1}(#2)}

\newcommand{\Res}{\textup{Res}}

\newcommand{\ZZ}{\mathbf{Z}}

\date{2011}

\renewcommand{\thefootnote}{\fnsymbol{footnote}}

\begin{document}

\begin{frontmatter}

\title{Distribution of Artin-Schreier Extensions}
\author{Thorsten Lagemann\footnote{
Universit\"at Paderborn, 
Institut f\"ur Mathematik, 33095 Paderborn, Germany. lagemann@math.uni-paderborn.de}}
\journal{Journal of Number Theory}

\begin{abstract}
The article at hand contains exact asymptotic formulas for the distribution of conductors of elementary abelian $p$-extensions of global function fields of characteristic~$p$. As a consequence for the distribution of discriminants, this leads to an exact asymptotic formula for simple cyclic extensions and an interesting lower bound for noncyclic elementary abelian extensions.  
\end{abstract}

\end{frontmatter} 

\setcounter{footnote}{0}
\renewcommand{\thefootnote}{\arabic{footnote}}

\section{Main results} 

Let $F$ be a global function field of characteristic~$p$, that is a transcendental extension of degree $1$ over some finite field $\ffield{q}$ of cardinality $q$, and let $\ffield{q}$ be algebraically closed in $F$. For an extension $E/F$, let $\mf{f}(E/F)$ be its relative conductor, and $\Norm{\mf{f}(E/F)} = q^{\deg \mf{f}(E/F)}$ its absolute norm. The object of interest is the counting function
$$ C(F,G;X) = |\{ E/F : \Gal{E}{F} \iso G,\, \Norm{\mf{f}(E/F)} \leq X \}| $$
of field extensions $E/F$ in a fixed algebraic closure of $F$ with given finite Galois group $G$ and bounded conductor. The values of this function coincides with finite coefficient sums of the Dirichlet series 
$$ \Phi(F,G;s) = \sum\limits_{\Gal{E}{F}\iso G} \Norm{\mf{f}(E/F)}^{-s}, $$
generated by the extensions $E/F$ with Galois group $G$. For the article at hand, we are interested in field extensions with elementary abelian Galois group of exponent $p$. These extensions are called Artin-Schreier extensions.

\begin{theorem}\label{THEOREMdirichlet} Let $F$ be a global function field of characteristic $p$ and $G=\cyclic{p}^r$ be the elementary abelian group with exponent $p$ and finite rank $r$. Let 
$$ \conda{G} = \frac{1 + (p-1)r}{p} $$ 
and
$$\condb{F,G} = \begin{cases} p-1 & r=1 \\ 1 & r>1. \end{cases} $$ 
Then the Dirichlet series $\Phi(F,G;s)$ has convergence abscissa $\conda{G}$, and it possesses a meromorphic continuation, such that $ \Phi(F,G;s) \prod\nolimits_{l=2}^p \zeta_F(ls-(l-1)r)^{-1} $ is holomorphic for $\Re(s)>\conda{G}-1/(2p)$, where $\zeta_F(s)$ denotes the zeta function of $F$. The periodic pole $s=\conda{G}$ with period $2\pi i/\log(q)$ has order $\condb{F,G}$, whereas the order of every other pole on the axis $\Re(s)=\conda{G}$ does not exceed $\condb{F,G}$. 
\end{theorem}

The series $\Phi(F,G;s)$ is a power series in $t=q^{-s}$, which is convergent for $|t|<q^{-\conda{G}}$, and its poles are located at $t=\xi q^{-\conda{G}}$ for finitely many roots of unity $\xi$. If $c_n$ denotes the $n$-th power series coefficient of $\Phi(F,G;s)$, we obtain $C(F,G;q^m)=\sum\nolimits_{n=0}^m c_n$. By an application of the Cauchy integral formula, which is elaborated in Theorem~\ref{A5}, this coefficient sum has the following asymptotic behaviour.

\begin{theorem}\label{THEOREMasymptoticCond} Let $F$ be a global function field of characteristic $p$ and $G=\cyclic{p}^r$ be the elementary abelian group with exponent $p$ and finite rank $r$. Then the number of Artin-Schreier extensions $E/F$ with Galois group $G$ has asymptotic
$$ C(F,G;X) \sim \condc{F,G} \, X^{\conda{G}} \, \log(X)^{\condb{F,G}-1} $$
with some explicitly computable constant $\condc{F,G}>0$.
\end{theorem}

The usage of asymptotic equivalence $f(X)\sim g(X)$ needs further comment. Usually this relation is defined by $\lim\nolimits_{n\rightarrow\infty} f(n)/g(n) = 1$. Here, in the function field case that is, the counting functions $f(X)$ are step functions in $X=q^n$, where $n$ runs over some arithmetic progression. Hence, the limit $\lim\nolimits_{n\rightarrow\infty} f(n)/g(n)$ would not exist for continuous functions $g(X)$. Therefore, I write $f(X)\sim g(X)$ by abuse of notation, 
if there are integers $\ell\geq 1$, $e\geq 0$ with $\lim\nolimits_{n\rightarrow\infty} f(q^{\ell n+e})/g(q^{\ell n+e}) = 1$. Analogously, I will make use of $f(X)\in\grossOmega{g(X)}$ for $g(q^{\ell n+e})\in\grossO{f(q^{\ell n+e})}$, and $f(X)\in\grossTheta{g(X)}$ for $f(q^{\ell n+e})\in\grossOmega{g(q^{\ell n+e})}\cap\grossO{g(q^{\ell n+e})}$.
This redefinition is useful if one likes to compare function field asymptotics with number field asymptotics. Note that the constant $\condc{F,G}$ in above theorem is not unique, varying by some $q$-powers, which depend on the chosen arithmetic progression. Disregarding this dichotomy, we obtain following formulas to compute $\condc{F,G}$ for some arithmetic progression.

\begin{auxtheorem}\label{THEOREMp2} For $p=2$ and $G=\cyclic{2}^r$, we have
$$ \condc{F,G} = \frac{|G|}{|\Aut{G}|}\, \frac{\log(q)}{1-q^{-(r+1)}}\,  \frac{\zeta_F(1)}{\zeta_F(r+1)} $$
where $\zeta_F(1)$ denotes the residue $\Res_{s=1}(\zeta_F(s))$.
\end{auxtheorem}

\begin{auxtheorem}\label{THEOREMr1} For $p\neq 2$ and $G=\cyclic{p}$, we have
$$ \condc{F,G} = \frac{1}{(p-2)!}\, \frac{|G|}{|\Aut{G}|}\, \frac{\log(q)}{1-q^{-1}}\, \frac{\zeta_F(1)^{p-1}}{p!}\, \prod\limits_{\mf{p}} \big(1+(p-1) \, \N{p}^{-1}\big)\, \big(1-\N{p}^{-1}\big)^{p-1}    $$
where the product runs over the primes of $F$. 
\end{auxtheorem}

Formulas for $\condc{F,\cyclic{p}^r}$ with $p\neq 2$, $r\neq 1$ are within reach, but omitted, as they are less handy. Theorem~\ref{THEOREMasymptoticCond} has also consequences on the distribution of discriminants $\mf{d}(E/F)$ of Artin-Schreier extensions $E/F$, which shall be elaborated and be compared to other known results in the second section. Particularly Theorem~\ref{THEOREMasymptoticDisc} and Theorem~\ref{THEOREMasymptoticDiscBound} are very interesting with regard of the Malle conjecture, which is also introduced in the second section.  After introducing some required notations and facts on class field theory, the proof of Theorem~\ref{THEOREMdirichlet} starts with the fourth section, where the $p$-rank of ray class groups is computed. This leads to a nice local-global principle for Artin-Schreier extensions, and an explicit formula for the number $c_{\mf{m}}$ of extensions with given conductor $\mf{m}$ in the fifth section. In the sixth section then, we are able to give an explicit description of the analytic behaviour of the Dirichlet series $ \Phi(F,G;s)$, which is generated by $c_{\mf{m}}$. This leads to a proof of Theorem~\ref{THEOREMdirichlet}. In the last section, Theorem~\ref{THEOREMasymptoticCond} and its addenda are proven, where it is the harder part to determine the constant $\condc{F,G}$. Finally, the appendix contains a  tauberian theorem for power series.

\section{Consequences on the distribution of discriminants}

For given global function field $F$ and finite group $G$, consider the number of field extensions $E/F$ in a fixed algebraic closure of $F$ with Galois group $G$ and bounded discriminant $\Norm{\mf{d}(E/F)} = q^{\deg \mf{d}(E/F)}$, that is the function
$$ Z(F,G;X) = |\{ E/F : \Gal{E}{F} \iso G,\, \Norm{\mf{d}(E/F)} \leq X \}|. $$ 
For the elementary abelian group $G=\cyclic{p}^r$, let
$$ a_p(G) = \frac{\conda{G}}{p^r-1} = \frac{1 + (p-1)r}{p(p^r-1)}.  $$
Then we obtain the following theorems as corollaries of Theorem~\ref{THEOREMasymptoticCond} and the F\"uhrerdiskriminantenformel.\footnote{E.g. see formula~(1.1) in \cite{Wright1989}.}

\begin{theorem}\label{THEOREMasymptoticDisc} Let $F$ be a global function field of characteristic~$p$ and $G=\cyclic{p}$ be the simple cyclic group of order $p$. Then the number of Artin-Schreier extensions $E/F$ with Galois group $G$ has asymptotic
$$ Z(F,G;X) \sim c(F,G) \, X^{a_p(G)} \, \log(X)^{\condb{F,G}-1} $$
with $c(F,G)=\condc{F,G}/(p-1)^{\condb{F,G}-1}$.
\end{theorem}

\begin{proof} By the F\"uhrerdiskriminantenformel, a $\cyclic{p}$-extension with conductor $\mf{f}$ has discriminant $\mf{d}=\mf{f}^{p-1}$. We hence obtain the desired asymptotic formula from Theorem~\ref{THEOREMasymptoticCond} and
$$ Z(F,\cyclic{p};X) = C(F,\cyclic{p};X^{1/(p-1)}). $$
Note that the factor $1/(p-1)^{\condb{F,G}-1}$ of $c(F,G)$ originates from $\log(X^{1/(p-1)})=\log(X)/(p-1)$.  
\end{proof}

\begin{theorem}\label{THEOREMasymptoticDiscBound}
Let $F$ be a global function field of characteristic~$p$ and $G=\cyclic{p}^r$ be the elementary abelian group with exponent $p$ and finite rank $r$. Then the number of Artin-Schreier extensions $E/F$ with Galois group $G$ has asymptotic
$$ Z(F,G;X) \in \grossOmega{X^{a_p(G)}}. $$
\end{theorem}

\begin{proof} By the F\"uhrerdiskriminantenformel, a $\cyclic{p}^r$-extension with conductor $\mf{f}$ has discriminant $\mf{d}\mid \mf{f}^{m}$ with $m=p^r-1$, where equality corresponds to the case, in which every cyclic subextension has conductor $\mf{f}$. We hence obtain 
$$ Z(F,\cyclic{p}^r;X) \geq C(F,\cyclic{p}^r;X^{1/m}) \in \grossOmega{X^{a_p(\cyclic{p}^r)}}, $$
by Theorem~\ref{THEOREMasymptoticCond}.
\end{proof}

In the same matter, we are able to give an upper bound for $Z(F,G;X)$, too. Therefore, let
$$ d_p(G) = \frac{\conda{G}}{p^r-p^{r-1}} = a_p(G)\, \frac{p^r-p^{r-1}}{p^r-1} $$
for the elementary abelian group $G=\cyclic{p}^r$.

\begin{theorem}\label{THEOREMasymptoticDiscBoundUpper}
Let $F$ be a global function field of characteristic~$p$ and $G=\cyclic{p}^r$ be the elementary abelian group with exponent $p$ and finite rank $r$. Then the number of Artin-Schreier extensions $E/F$ with Galois group $G$ has asymptotic
$$ Z(F,G;X) \in \grossO{X^{d_p(G)}}. $$
\end{theorem}

\begin{proof} Let $E/F$ be a $\cyclic{p}^r$-extension with conductor $\mf{f}$ and discriminant $\mf{d}$, and let $E_1,\ldots,E_r$ be a minimal set of cyclic subextensions with $E=E_1\cdots E_r$. Since $E/F$ has conductor $\mf{f}$, one of the generating subfields has conductor $\mf{f}$, too. Let $E_1$ be that field. By the F\"uhrerdiskriminantenformel, we now obtain $\mf{f}^{m}\mid\mf{d}$ with $m=p^r-p^{r-1}$, where equality corresponds to the case, in which $E_2,\ldots,E_r$ are unramified. Hence, every $\cyclic{p}^r$-extension with discriminant norm $X$ has conductor norm $X^{1/m}$ at most. This results in
$$ Z(F,\cyclic{p}^r;X) \leq C(F,\cyclic{p}^r;X^{1/m}) \in \grossO{X^{d_p(\cyclic{p}^r)}}, $$ 
by Theorem~\ref{THEOREMasymptoticCond}.
\end{proof} 

As the number of extensions with given conductor grows at least linearly with its norm, it is more likely that all generating cyclic subextensions have the same conductor than that all but one are unramified. That means, the situation as considered for the lower bound happens far more often than the situation as considered for the upper bound. Hence, one might ask if $a_p(G)$ already gives the $X$-exponent of $Z(F,G;X)$. Lets now compare the results with the distribution of abelian groups in general. For the simple cyclic group, the formula of $c(F,\cyclic{p})$ gives a comparable function field result to the number field formula of Henri Cohen et al. \cite{CohenDiazOlivierDDCEPD}, which is not completely analogous though. For abelian groups with order coprime to the characteristic, we obtain the following theorem by the results of David Wright \cite{Wright1989}. His work only lacks a handy asymptotic formula for $Z(F,G;X)$ in the function field case, which we get by the redefinition of asymptotic equivalence as described in the beginning of this article.

\begin{theorem*}[Wright \cite{Wright1989}]
Let $F$ be a global function or number field of characteristic~$p\geq 0$ and $G$ be a finite abelian group. Let $\ell$ be the smallest prime divisor of $|G|$, $G[\ell]$ the $\ell$-torsion group of $G$, and $\tilde{F}$ the splitting field of $X^{\ell}-1$ over $F$. Then we define
$$ a(G) = \frac{\ell}{|G|  (\ell-1)}, \qquad b(F,G) = \frac{|G[\ell]| - 1}{\fdindex{\tilde{F}}{F}}. $$
Now assume $p\nmid |G|$. Then there is a positive constant $c(F,G)$, such that the number of $G$-extensions $E/F$ has asymptotic
$$ Z(F,G;X) \sim c(F,G) \, {X^{a(G)}\, \log(X)^{b(F,G)-1}}. $$
\end{theorem*}

\begin{proof} In the number field case, this is just Theorem I.2 loc. cit. In the function field case, this is the consequence of Theorem I.1 loc. cit., the discussion before Theorem I.3 loc. cit., and a tauberian theorem like Theorem \ref{A5}.
\end{proof}

This result is restricted to the case $p\nmid |G|$. David Wright conjectured, that the case $p \mid |G|$ should behave analogously, but omitted the calculations as they are more involved. His formulas do indeed extend for the simple cyclic group $\cyclic{p}$, and might also extend for the Klein four group $\cyclic{2}^2$ in characteristic $2$.  But they do not hold for any other noncyclic elementary abelian group $\cyclic{p}^r$, as the lower bound in Theorem~\ref{THEOREMasymptoticDiscBound} breaks the constant $a(G)$. This follows from the next proposition.

\begin{prop} We have $a_p(\cyclic{p}^r)= a(\cyclic{p}^r)$ for $\cyclic{p}^r=\cyclic{p},\cyclic{2}^2$, and $ a_p(\cyclic{p}^r) > a(\cyclic{p}^r) $ otherwise. 
\end{prop}
\begin{proof} Write $a_p(\cyclic{p}^r)-a(\cyclic{p}^r)$ as a fraction by extending $a_p(\cyclic{p}^r)$ with the denominator of $a(\cyclic{p}^r)$ and vice versa. Then this fraction has the numerator
$$ \big( (r-1)\, (p-1)^2 - p \big) \cdot p^{r-1} + p \geq 0,  $$
which leads to the stated proposition. 
\end{proof}

This will have consequences on possible extensions of the Malle conjecture for global function fields. 

\begin{EMC}[Extended Malle Conjecture] Let $F$ be global field of characteristic $p\geq 0$ and $G$ be a finite group. Then we conjecture, that there are positive constants $a(F,G)$, $b(F,G)$, $c(F,G)$, such that
$$ Z(F,G;X) \sim c(F,G) \, X^{a(F,G)} \, \log(X)^{b(F,G)-1}. $$ 
What are the formulas of these constants? 
\end{EMC}

Gunter Malle \cite{Malle2002, Malle2004} proposed formulas for $a(F,G)=a(G)$ and $b(F,G)$ in the number field case, which extend Wright's formulas for abelian groups and other known results. While his conjecture for $b(F,G)$ needs some adjustments \cite{JK2005, Tuerk2008}, the formula for $a(F,G)=a(G)$, which remarkably does not depend on $F$, seems to be consistent for both number and function fields~$F$ with characteristic~$p$ coprime to the group order~$|G|$. For function fields, this was shown by Jordan Ellenberg and Akshay Venkatesh \cite{Ellenberg2005} heuristically. With my results at hand, it is plausible that a general formula for $a(F,G)$ including the case $p \mid |G|$ is not independent of $F$ and should at least depend on its characteristic~$p$, a formula like $a(F,G)=a_p(G)$ for example. 
The main specialities of the case $p\mid |G|$ are, that every prime $\mf{p}$ of $F$ can wildly ramify in some $G$-extension $E/F$, and it is possible to have infinitely many local $G$-extensions $E/\local{F}{p}$ over any localisation~$\local{F}{p}$. Both effects do not occur in the case $p\nmid |G|$, including number fields.

\section{Notations and preliminaries} Let $F$ denote a global function field with constant field $\ffield{q}$ of cardinality $q$, algebraically closed in $F$. For a prime $\mf{p}$ of $F$, let $\ffield{\mf{p}}\geq \ffield{q}$ denote the residue class field, and $\N{p} = |\ffield{\mf{p}}|$ its norm. The valuation of $\mf{p}$ is denoted by $\val{p}$.  Let $g_F$ be the genus, $h_F$ the class number, and $\Cl$ the class group of $F$. In particular, we have $\Cl\iso\ZZ \times \Cl[h_F]$. For a module $\mf{m}$, an effective divisor of $F$ that is, let $F^{\mf{m}}$ be the multiplicative group of functions $x\in F^{\times}$ with divisor coprime to $\mf{m}$, and let $F_{\mf{m}}\leq F^{\mf{m}}$ be the ray mod $\mf{m}$ of functions $x\in F^{\mf{m}}$ with $x\in\open{p}{\ord{p}{\mf{m}}}$ for all $\mf{p}\mid\mf{m}$.  
The ray class group $\Cl_{\mf{m}}$  mod $\mf{m}$ is defined by the exact sequence of abelian groups
\begin{equation}\label{RayClassGroupSequence}
  1 \rightarrow \ffield{q}^{\times} \rightarrow F^{\mf{m}}/F_{\mf{m}} \rightarrow \Cl_{\mf{m}} \rightarrow \Cl \rightarrow 1. 
\end{equation}
By class field theory, there is a bijective map of finite abelian extensions $E/F$ with module $\mf{m}$ and subgroups $U\leq\Cl_{\mf{m}}$ of finite index, such that $\Gal{E}{F} \iso \Cl_{\mf{m}}/U$.\footnote{E.g. see page 139 and Theorem 9.23, page 140 in \cite{RosenNTFF} for the sequence (\ref{RayClassGroupSequence}) and the Artin reciprocity law respectively.} Let $U_{\mf{m}}$ be the following finite $p$-group\footnote{E.g. see section 9.2 in \cite{Hess2003} for the isomorphisms.} 
\begin{equation}\label{EinseinheitenGroup} 
U_{\mf{m}} = (F^{\mf{m}}/F_{\mf{m}}) / (F^{\mf{m}}/F_{\mf{m}})^p \iso \prod\limits_{\mf{p}^m\mid\mid\mf{m}} U_{\mf{p}^m} \iso \prod\limits_{\mf{p}^m\mid\mid\mf{m}} \Open{p}{}/\langle\Open{p}{m},\Open{p}{}^p\rangle,
\end{equation}
where the product is to read as follows. The relation $\mf{p}^m\mid\mid\mf{m}$ stands for $\ord{p}{\mf{m}}=m > 0$. For an expression $f_{\mf{p}}(m)$, the product $\prod\nolimits_{\mf{p}^m\mid\mid\mf{m}} f_{\mf{p}}(m)$ stands for $\prod\nolimits_{\mf{p}\mid\mf{m}} f_{\mf{p}}(\val{p}(\mf{m}))$. For example, let $f_{\mf{p}}(m)=\mf{p}^{m-1}$ and we have $\prod\nolimits_{\mf{p}^m\mid\mid\mf{m}} \mf{p}^{m-1} = \prod\nolimits_{\mf{p}\mid\mf{m}} \mf{p}^{\val{p}(\mf{m})-1}$, which is still an effective divisor. The same definition holds for sums. For example, the divisor above has degree $\sum\nolimits_{\mf{p}^m\mid\mid\mf{m}} (m-1)\deg\mf{p} = \sum\nolimits_{\mf{p}\mid\mf{m}} (\ord{p}{\mf{m}}-1)\deg \mf{p} \geq 0$. For the trivial module $\mf{m}=\mf{1}$, these products and sums are particularly defined as empty ones. Further, let $S = \{ x\in F^{\times} : (x)=\mf{a}^p \textup{ for some divisor } \mf{a} \}/F^p$ be the $p$-Selmer group, and let $S_{\mf{m}}=\{ [x]\in S : x\in F_{\mf{m}}\}$ be the Selmer ray group mod $\mf{m}$ of $F$.

\section{Artin-Schreier extensions with given module}

In this section, we shall compute the order of the ray class group $\Cl_{\mf{m}}/\Cl_{\mf{m}}^p$. It nearly behaves multiplicatively over the prime decomposition of $\mf{m}$, if there were not the Selmer ray group $S_{\mf{m}}$, whose order is not multiplicative. But as we will see, $S_{\mf{m}}$ is trivial for almost all modules $\mf{m}$ of interest. Proposition~\ref{prop1} is the function field analogue of Proposition 2.12 in \cite{CohenDiazOlivierDDCEPD}, and its proof is essentially the same. 

\begin{prop}\label{prop1} We have the exact sequence of finite abelian $p$-groups
$$ 1 \rightarrow S_{\mf{m}} \rightarrow \Cl[p] \rightarrow U_{\mf{m}} \rightarrow \Cl_{\mf{m}}/\Cl_{\mf{m}}^p \rightarrow \Cl/\Cl^p \rightarrow 1. $$
\end{prop}
\begin{proof} The sequence (\ref{RayClassGroupSequence}) remains right exact by tensoring with $\cyclic{p}$, and we obtain
$$ U_{\mf{m}} \rightarrow \Cl_{\mf{m}}/\Cl_{\mf{m}}^p \rightarrow \Cl/\Cl^p \rightarrow 1. $$ 
The kernel $U\leq U_{\mf{m}}$ of the first map is generated by functions $x\in F^{\mf{m}}$ with divisor $(x)=(y)\, \mf{a}^p$ and $y\in F_{\mf{m}}$. Therefore, any class $[x] = [xy^{-1}]$ in $U$ is generated by a Selmer element $[z]\in S$ with $(z) = \mf{a}^p$. Conversely,  every Selmer class $[x]\in S$ can be mapped in $U$ by the approximation theorem,\footnote{E.g. see Theorem 1.3.1, page 12 in \cite{StiALF}.} and this map has kernel $S_{\mf{m}}$. Finally, we use the following isomorphism of the $p$-torsion group $\Cl[p]$ with the Selmer group $S$. For a given divisor class $[\mf{a}]\in \Cl[p]$, there is a function $z\in F^{\times}$ with divisor $(z) = \mf{a}^p$, and vice versa. Then the map $[\mf{a}] \mapsto [z]$ is well defined and provides an isomorphism $\Cl[p]\iso S$. 
\end{proof}

\begin{lemma}\label{prop2} For any integer $m\geq 1$, let $r_m=m-1-\lfloor (m-1)/p \rfloor$, and $r_0=0$. Let $\mf{m}$ be a module. Then we have  
$$ \grindex{\Cl_{\mf{m}}}{\Cl_{\mf{m}}^p} = p \, |S_{\mf{m}}| |U_{\mf{m}}| = p \, |S_{\mf{m}}|  \prod\limits_{\mf{p}^m\mid\mid\mf{m}}
\N{p}^{r_m}. $$
\end{lemma}
\begin{proof} With the exact sequence of Proposition~\ref{prop1} and the isomorphism in (\ref{EinseinheitenGroup}), we obtain 
$$ \grindex{\Cl_{\mf{m}}}{\Cl_{\mf{m}}^p} = \frac{|S_{\mf{m}}| |U_{\mf{m}}|  \grindex{\Cl}{\Cl^p}}{|\Cl[p]|} = p\, |S_{\mf{m}}|\, |U_{\mf{m}}| = p \, |S_{\mf{m}}| \prod\limits_{\mf{p}^m||\mf{m}} |U_{\mf{p}^m}|. $$
By the well-known Eins\-einheiten\-satz\footnote{E.g. see Proposition 6.2, page 18 in \cite{Fesenko2002}.}, the $p$-rank of $U_{\mf{p}^m}$ is $\fdindex{\ffield{\mf{p}}\negthinspace}{\negthinspace\ffield{p}} \, r_m$, and we obtain the proposed index formula.  
\end{proof}

\begin{cor}\label{prop3} Let $\mf{m}=\mf{np}$ be a module with prime divisor $\mf{p}$ coprime to $\mf{n}$. Then $|S_{\mf{m}}| = |S_{\mf{n}}|$ holds.  
\end{cor}
\begin{proof} Any ramified $p$-extension of the local field $\local{F}{p}$ is wildly ramified, whence its local conductor cannot be squarefree. This implies that an abelian $p$-extension of $F$ with module $\mf{m}$ cannot be ramified in~$\mf{p}$, and abelian $p$-extensions of $F$ with conductor $\mf{m}$ are nonexistent. Hence, the ray class factor groups $\Cl_{\mf{m}}/\Cl_{\mf{m}}^p$ and $\Cl_{\mf{n}}/\Cl_{\mf{n}}^p $ are isomorphic, and the proposed coincidence of cardinality follows from Lemma~\ref{prop2} and $|U_{\mf{p}}|=1$. 
\end{proof}

The order of Selmer ray groups is hard to compute in general. For modules with a large square divisor however, we can utilise the embedding of the Selmer group in the group of regular logarithmic differentials of $F$.

\begin{lemma}\label{prop4} Let $\mf{m}$ be a module with the condition 
$$ \sum\limits_{\mf{p}^m\mid\mid\mf{m}} (m-1) \, \deg\mf{p} > 2g_F - 2. $$
Then the Selmer ray group $S_{\mf{m}}$ is trivial. 
\end{lemma}
\begin{proof} Let $\Omega_F$ be the group of differentials and
$$ \psi : S_{\mf{m}} \rightarrow \Omega_F, \quad [x] \mapsto \frac{dx}{x}. $$
By the product rule, this map is well defined and homomorphic via
$$ \frac{d(xy)}{xy} = \frac{ydx + xdy}{xy} = \frac{dx}{x} + \frac{dy}{y} \qquad\textup{and}\qquad \frac{dz^p}{z^p} = p\cdot
\frac{dz}{z} = 0 $$
for $x,y,z\in F^{\times}$. The differential $dx$ is zero if and only if $x=z^p$ is a nonseparating element of
$F$.\footnote{E.g. see Proposition 3.10.2 (d), page 144, and Proposition 4.1.8 (c), page 160 in \cite{StiALF}.}  Therefore, $\psi$ is injective.
We now suppose $x\in F_{\mf{m}}$ to generate a nontrivial Selmer class in $S_{\mf{m}}$, and denote its logarithmic differential with $\omega = dx/x$. For any prime $\mf{p}$ of $F$ with uniformizing element $t$, there is an integer $i = -\val{p}(x)/p$, such that $u=xt^{ip}$  is a unit modulo $\mf{p}$. We hence obtain $\omega = du/u$ by the product rule, and $\omega = (du/dt)(dt/u)$ by the chain rule. This yields 
$$ \val{p}(\omega) = \val{p}(du/dt) - \val{p}(u) + \val{p}(dt) = \val{p}(du/dt).   $$
Observe, that $\val{p}(dt)=0$ follows from $\val{p}(t)=1$.\footnote{E.g. see Theorem 4.3.2 (e), page 171 in \cite{StiALF}.}
The series expansion of $u = \sum a_nt^n$ with respect to $t$ contains at least one nontrivial term $a_nt^n$ with $p\nmid n$, as $u$ is no $p$-th power. If $n$ is chosen to be minimal with this property, we have $n\geq 1$ and $du/dt \in na_nt^{n-1}\Open{p}{}$, whence $\omega$ has positive valuation $\val{p}(\omega)=n-1$. As we have chosen $\mf{p}$ to be arbitrary, $\omega$ is hence regular, and has a positive divisor of degree $2g_F-2$.\footnote{E.g. see Corollary 1.5.16, page 31 in \cite{StiALF}.} This already leads to a contradiction in case of $g_F=0$, of course, but more generally, we yield the inequality
$$ 2g_F-2 = \deg(\omega) \geq \sum\limits_{\mf{p}\mid\mf{m}} \ord{p}{\omega} \, \deg\mf{p}. $$ 
Since $x$ is chosen from the ray $F_{\mf{m}}$, we have $x\in\Open{p}{m}$ for $\mf{p}^m\mid\mid\mf{m}$. Then the integer $n$, obtained by the procedure above, fulfils $n\geq m$, which implies $ \val{p}(\omega) \geq m - 1$ and 
$$ 2g_F-2 \geq \sum\limits_{\mf{p}^m\mid\mid\mf{m}} (m-1)\, \deg\mf{p}. $$
Since $\mf{m}$ is a module violating this inequality, we yield the desired contradiction to the existence of $x$. Hence, the Selmer ray group $S_{\mf{m}}$ only contains the trivial class.
\end{proof}

\section{Artin-Schreier extensions with given conductor}

With the results of the last section, we easily obtain the number of Artin-Schreier extensions with given module. If we like to count extensions with given conductor, we are led to an inclusion-exclusion principle via the M\"obius formula. 

\begin{prop}\label{prop3-1} For any module $\mf{m}$, let $c_{\mf{m}}$ be the number of $\cyclic{p}^r$-extensions $E/F$ with conductor $\mf{m}$. Further, let
$$ e(X) = \prod\limits_{i=0}^{r-1} \frac{pX-p^i}{p^r-p^i} = \sum\limits_{i=0}^r e_i X^i. $$
Then we have $c_{\mf{1}} = e(|\Cl[p]|)$ for the trivial module $\mf{m}=\mf{1}$, and
$$ c_{\mf{m}} =  \sum\limits_{i=1}^r e_i  \sum\limits_{\mf{n}\mid\mf{m}} \mu(\mf{m}\mf{n}^{-1})  |S_{\mf{n}}|^i  |U_{\mf{n}}|^i $$
for nontrivial modules $\mf{m}\neq \mf{1}$, where $\mu(\mf{n})$ denotes the M\"obius function for modules. 
\end{prop}
\begin{proof} For any elementary abelian group $A$, the number of subgroups $U\leq A$ with factor group $A/U\iso\cyclic{p}^r$ is given by $e(|A|/p)$.\footnote{E.g. see Hilfssatz 8.5, page 311 in \cite{HupA}.} Note that we do not need any assumptions on the $p$-rank of $A$, since $e(p^i)=0$ hold for $0\leq i\leq r-2$. Hence, by Lemma~\ref{prop2}, the number of all $\cyclic{p}^r$-extensions $E/F$ with module $\mf{m}$ sums up to 
$$ e(|S_{\mf{m}}| |U_{\mf{m}}|) = \sum\limits_{\mf{n}\mid\mf{m}} c_{\mf{n}}. $$
Now, $c_{\mf{1}} = e(|\Cl[p]|)$ follows for the trivial module $\mf{m}=\mf{1}$, since we have $S_{\mf{1}}= S \iso \Cl[p]$ and $U_{\mf{1}}=1$. Further, we obtain 
$$ c_{\mf{m}} = \sum\limits_{\mf{n}\mid\mf{m}} \mu(\mf{m}\mf{n}^{-1}) e(|S_{\mf{n}}| |U_{\mf{n}}|) = \sum\limits_{i=0}^r e_i  \sum\limits_{\mf{n}\mid\mf{m}} \mu(\mf{m}\mf{n}^{-1})  |S_{\mf{n}}|^i  |U_{\mf{n}}|^i  $$
by the inversion formula of M\"obius. The sum $\sum\nolimits_{\mf{n}|\mf{m}} \mu(\mf{n})$ vanishes for nontrivial modules $\mf{m}\neq \mf{1}$, whence the formula for $c_{\mf{m}}$ only depends on the indices $i=1,\ldots,r$. 
\end{proof}

With the M\"obius formula for $c_{\mf{m}}$, we near a local-global principle for the number of Artin-Schreier extensions. For example, suppose that the Selmer group $S$ of $F$ was trivial. Then all Selmer ray groups are trivial, and we would obtain
$$ c_{\mf{m}} = \sum\limits_{i=1}^r e_i \sum\limits_{\mf{n}\mid\mf{m}} \mu(\mf{m}\mf{n}^{-1}) |U_{n}|^i = \sum\limits_{i=1}^r e_i \prod\limits_{\mf{p}^m\mid\mid\mf{m}} \left( |U_{\mf{p}^m}|^i - |U_{\mf{p}^{m-1}}|^i \right). $$
The second equation, which we later shall prove in detail, is just a consequence of the multiplicativity of $\mu(\mf{n})$ and $|U_{\mf{n}}|^i$. By Lemma~\ref{prop4}, we might hope for formulas like the above, since the Selmer ray group $S_{\mf{m}}$ is trivial for modules $\mf{m}$ with a large square divisor. But we must carefully regard the Selmer ray groups $S_{\mf{n}}$ for divisors $\mf{n}\mid\mf{m}$, too. With the following lemma, we shall see, that a huge family of modules allows above formula, whereas formulas for the infinite complementary family depends only on finitely many modules. This yields a crucial tool for computing $\Phi(F,\cyclic{p}^r;s)$ in the next section. 
Call a module $\mf{m}$ squareful, if every prime divisor $\mf{p}\mid\mf{m}$ has at least order $\val{p}(\mf{m})\geq 2$. Let $\mc{M}$ be the finite set consisting of the trivial module $\mf{m}=\mf{1}$ and the squareful modules $\mf{m}$ supported only by primes $\mf{p}$ with $\deg\mf{p} \leq 2g_F-2$ and $\ord{p}{\mf{m}}\leq 2g_F$. Note that $\mc{M} $ consists only of the trivial module in case of $g_F=0,1$.

\begin{lemma}\label{prop3-2} The number $c_{\mf{m}}$ of $\cyclic{p}^r$-extensions $E/F$ with conductor $\mf{m}$ fulfils the following formulas.
\begin{enumerate}
 \item[\textup{(a)}] Let $\mf{m}=\mf{1}$ be the trivial module. Then $c_{\mf{1}}= e(|\Cl[p]|)$ holds. 
 \item[\textup{(b)}] Let $\mf{m}=\mf{m}_0\mf{m}_1^2$ with $\mf{m}_0\in\mc{M}$ and $\mf{m}_1$ being a squarefree module supported by primes $\mf{p}$ of degree $\deg\mf{p}>2g_F-2$ only. Then
$$ c_{\mf{m}}  =  \mu(\mf{m}_1)  \tilde{c}_{\mf{m}_0} + \sum\limits_{i=1}^r e_i  \prod\limits_{\mf{p}^m\mid\mid\mf{m}}
\left(\N{p}^{i r_m}-\N{p}^{i r_{m-1}}\right)
$$
holds with $ \tilde{c}_{\mf{1}} = c_{\mf{1}} - \sum\nolimits_{i=0}^r e_i $ for the trivial module $\mf{m}_0 = \mf{1}$ and 
 $$ \tilde{c}_{\mf{m}_0} = c_{\mf{m}_0} - \sum\limits_{i=1}^r e_i \prod\limits_{\mf{p}^m\mid\mid\mf{m}_0}
\left(\N{p}^{i r_m}-\N{p}^{i r_{m-1}}\right)  $$
otherwise.
 \item[\textup{(c)}] Let $\mf{m}$ be a module not covered by \textup{(a,b)}. Then   
$$ c_{\mf{m}} = \sum\limits_{i=1}^r e_i \prod\limits_{\mf{p}^m\mid\mid\mf{m}}
\left(\N{p}^{i r_m}-\N{p}^{i r_{m-1}}\right) $$
holds. 
In particular, we have $c_{\mf{m}} = 0$  if $\mf{m}$ is not squareful, or more generally, has a prime divisor $\mf{p}$ with valuation $\val{p}(\mf{m}) \equiv 1 \pmod{p}$.
\end{enumerate}
\end{lemma}
Note that this lemma does not provide any information on $c_{\mf{m}}$ for $\mf{m}\in\mc{M}\setminus\{\mf{1}\}$.
\begin{proof}  Assertion (a) is just a restatement of Proposition~\ref{prop3-1}. For the rest of the proof, we can consider $\mf{m}$ to be nontrivial. Let $a_{\mf{m}}$ be given by $a_{\mf{m}} = c_{\mf{m}} - b_{\mf{m}}$ and
$$ b_{\mf{m}} = \sum\limits_{i=1}^r e_i \prod\limits_{\mf{p}^m\mid\mid\mf{m}} \left(\N{p}^{i r_m}-\N{p}^{i r_{m-1}}\right). $$
Then assertions (b) and (c) are equivalent to $a_{\mathfrak{m}} = \mu(\mathfrak{m}_1) \tilde{c}_{\mathfrak{m}_0}$ and $a_{\mathfrak{m}}=0$ respectively. Note that $\tilde{c}_{\mathfrak{1}} = a_{\mathfrak{1}} - e_0$ and $\tilde{c}_{\mathfrak{m}_0} = a_{\mathfrak{m}_0}$ hold for $\mathfrak{m}_0 \neq \mathfrak{1}$. We start with the identity $\N{p}^{ir_m} = |U_{\mf{p}^m}|^i$ provided by Lemma~\ref{prop2}. Then 
$$ \N{p}^{ir_m} - \N{p}^{ir_{m-1}} = |U_{\mf{p}^m}|^i - |U_{\mf{p}^{m-1}}|^i = \sum\limits_{n=0}^m \mu(\mf{p}^{m-n}) |U_{\mf{p}^n}|^i $$
follows, as $\mu(\mf{n})$ vanishes on squares. The order $|U_{\mf{n}}|^i$ provides a multiplicative function for modules, whence we obtain 
$$ \prod\limits_{\mf{p}^m\mid\mid\mf{m}} \sum\limits_{n=0}^m \mu(\mf{p}^{m-n}) |U_{\mf{p}^n}|^i = \sum\limits_{\mf{n}\mid\mf{m}} \mu(\mf{m}\mf{n}^{-1}) \prod\limits_{\mf{p}^n\mid\mid\mf{n}} |U_{\mf{p}^n}|^i = \sum\limits_{\mf{n}\mid\mf{m}} \mu(\mf{m}\mf{n}^{-1}) |U_{\mf{n}}|^i. $$
Putting all together, this results in
\begin{equation}\label{MoebiusFormulaCaseC0}
 b_{\mf{m}} = \sum\limits_{i=1}^r e_i \sum\limits_{\mf{n}|\mf{m}} \mu(\mf{m}\mf{n}^{-1}) |U_{\mf{n}}|^i. 
\end{equation}
As we consider $\mf{m}$ to be nontrivial, we obtain 
\begin{equation}\label{MoebiusFormulaCaseC1}
  a_{\mf{m}} = c_{\mf{m}} - b_{\mf{m}} = \sum\limits_{i=1}^r e_i \sum\limits_{\mf{n}\mid\mf{m}} \mu(\mf{m}\mf{n}^{-1}) (|S_{\mf{n}}|^i-1) 
|U_{\mf{n}}|^i 
\end{equation}
by Proposition~\ref{prop3-1}. Now, let $\mf{m}$ be a module as considered in assertion~(c). Then there is a multiple prime divisor $\mf{p}^m\mid\mid\mf{m}$ with $m>2$, such that $\deg \mf{p} > 2g_F -2$ or $m>2g_F$ holds. Any divisor $\mf{n}\mid\mf{m}$ with $\mu(\mf{m}\mf{n}^{-1})\neq 0$ is divisible by $\mf{p}^{m-1}$, whence its Selmer ray group $S_{\mf{n}}$ is trivial by Lemma~\ref{prop4}. This yields $a_{\mf{m}} = 0$ by formula~(\ref{MoebiusFormulaCaseC1}), as desired. Now, let $\mf{m} = \mf{m}_0\mf{m}_1^2$ be a module as considered in assertion~(b). Any divisor $\mf{n}\mid\mf{m}$ with $\mu(\mf{m}\mf{n}^{-1})\neq 0$ has necessarily the form $\mf{n}=\mf{n}_0\mf{m}_1\mf{n}_1$ with $\mf{n}_0\mid\mf{m}_0$ and $\mf{n}_1\mid\mf{m}_1$, as $\mf{m}\mf{n}^{-1}$ would contain a square otherwise. Considering this form of $\mf{n}$, the Selmer ray group $S_{\mf{n}}$ is trivial for $\mf{n}_1\neq \mf{1}$ by Lemma~\ref{prop4}. Hence, the right hand side in formula~(\ref{MoebiusFormulaCaseC1}) only depends on modules $\mf{n}=\mf{n}_0\mf{m}_1$ with $\mf{n}_0 \mid \mf{m}_0$, and we are led to 
\begin{align*} a_{\mf{m}} & =  \sum\limits_{i=1}^r e_i \sum\limits_{\mf{n}_0\mid\mf{m}_0} \mu(\mf{m}\mf{n}_0^{-1}\mf{m}_1^{-1})(|S_{\mf{n}_0\mf{m}_1}|^i-1) |U_{\mf{n}_0\mf{m}_1}|^i 
\end{align*}
Since $\mf{m}_1$ is squarefree, we have $|U_{\mf{m}_1}|=1$ and $|S_{\mf{n}_0\mf{m}_1}|=|S_{\mf{n}_0}|$  by Lemma~\ref{prop2} and Corollary~\ref{prop3} respectively. Using the multiplicativity of $\mu(\mf{n})$ and $|U_{\mf{n}}|$, we now obtain
$$ a_{\mf{m}} = \mu(\mf{m}_1) \sum\limits_{i=1}^r e_i \sum\limits_{\mf{n}_0\mid\mf{m}_0} \mu(\mf{m}_0\mf{n}_0^{-1}) (|S_{\mf{n}_0}|^i-1) |U_{\mf{n}_0}|^i. $$
With the formula for $a_{\mf{m}_0}$ given by formula~(\ref{MoebiusFormulaCaseC1}), this implies the desired identity $ a_{\mf{m}} = \mu(\mf{m}_1) \tilde{c}_{\mf{m}_0} $ in the case $\mf{m}_0\neq\mf{1}$. If $\mf{m}_0=\mf{1}$ is the trivial module, we  obtain 
$$ a_{\mf{m}} = \mu(\mf{m}_1) \sum\limits_{i=1}^r e_i (|S|^i -1) = \mu(\mf{m}_1) \left( e(|S|) - e_0 - \sum\limits_{i=1}^r e_i \right) = \mu(\mf{m}_1) \tilde{c}_{\mf{1}}  $$
as a consequence of assertion~(a). 
\end{proof}

\section{Dirichlet series of Artin-Schreier extensions}

This section contains a rather complete description of the Dirichlet series
$$ \Phi(F,\cyclic{p}^r;s) = \sum\limits_{\Gal{E}{F} \iso \cyclic{p}^r} \Norm{\mf{f}(E/F)}^{-s} = \sum\limits_{\mf{m}} c_{\mf{m}} \N{m}^{-s}, $$  
generated by the $\cyclic{p}^r$-extensions $E/F$ and by the modules $\mf{m}$ of $F$ respectively. As a consequence of Lemma~\ref{prop3-2}, the series $\Phi(F,\cyclic{p}^r;s)$ can be decomposed in a finite sum of Euler products $\Phi_i(s)$ and an error term $\Upsilon_r(s)$ with
\begin{equation}\label{Phi}
 \Phi_i(s) = \prod\limits_{\mf{p}} \bigg( 1 + (\N{p}^i-1) \sum\limits_{p\, \nmid\, n} \N{p}^{ir_n-(n+1)s} \bigg). 
\end{equation}
 
\begin{lemma}\label{prop4-1} The Dirichlet series $\Phi(F,\cyclic{p}^r;s)$ has the decomposition 
$$ \Phi(F,\cyclic{p}^r;s) = \sum\limits_{i=1}^r e_i \Phi_i(s) + \Upsilon_r(s), $$
where $\Upsilon_r(s)$ is a rational function in $t=q^{-s}$ without poles for $\Re(s)>1/4$.  
\end{lemma}
\begin{proof} Let $\Phi_{i,\mf{p}}(s)$ be the Euler factor of $\Phi_i(s)$ corresponding to the prime $\mf{p}$. By definition in Lemma \ref{prop2}, we have $r_n=n-1-\lfloor (n-1)/p\rfloor$. Since $r_{n+1}=r_n$ holds for $p\mid n$, and $r_{n+1}=1+r_n$ otherwise, it is easy to verify
$$ \Phi_{i,\mf{p}}(s) = 1 + (\N{p}^i-1) \sum\limits_{p\, \nmid\,  n} \N{p}^{ir_n-(n+1)s} = 1 + \sum\limits_{m\geq 1} (\N{p}^{ir_m}-\N{p}^{ir_{m-1}})\N{p}^{-ms} $$
by the transformation $m=n+1$. We hence obtain 
$$ \sum\limits_{i=1}^r e_i \Phi_i(s) = \sum\limits_{\mf{m}}  \left( \sum\limits_{i=1}^r e_i \prod\limits_{\mf{p}^{m}\mid\mid\mf{m}}
\left(\N{p}^{ir_m}-\N{p}^{ir_{m-1}}\right) \right) \N{m}^{-s} . $$
The coefficient of the latter Dirichlet series corresponding to $\mf{m}$ coincides with $c_{\mf{m}}$, if $\mf{m}$ is of the form as considered in Lemma~\ref{prop3-2} (c). Therefore, the difference $\Upsilon_r(s) = \Phi(F,\cyclic{p}^r;s) - \sum\nolimits_{i=1}^r e_i\Phi_i(s)$ is generated by the modules $\mf{m}$ as considered in Lemma~\ref{prop3-2} (a,b). Then we obtain
$$ \Upsilon_r(s) = \left( c_{\mf{1}} - \sum\limits_{i=1}^r e_i \right) + \left( \sum\limits_{\mf{m}_0\in\mc{M}} \tilde{c}_{\mf{m}_0} \Norm{\mf{m}_0}^{-s} \prod\limits_{\deg \mf{p}>2g_F-2} (1-\N{p}^{-2s}) \right), 
$$
where the big parentheses correspond to case (a) and (b) respectively. Note hereby, that the restricted Euler product $\prod (1 - \N{p}^{-2s})$ equals the restricted series $\sum \mu(\mf{m}_1) \Norm{\mf{m}_1}^{-2s}$ generated by the modules $\mf{m}_1$ as given in (b). Since $\mc{M}$ is finite, the sum over its elements is a polynomial $f(q^{-s})$ in $q^{-s}$. Furthermore, there are only finitely many primes $\mf{p}$ of degree $\deg\mf{p}\leq 2g_F-2$, and the product $g(q^{-s})$ over their Euler factors $(1-\N{p}^{-2s})$ is a polynomial as well. Consequently, $\Upsilon_r(s)$ is rational in $q^{-s}$ with formula
\begin{equation}\label{Upsilon} \Upsilon_r(s) = c_{\mf{1}} - \sum\limits_{i=1}^r e_i + \frac{f(q^{-s})}{g(q^{-s})} \, \zeta_F(2s)^{-1}, 
\end{equation}
where $\zeta_F(s)$ is the zeta function of $F$. The poles of $\Upsilon_r(s)$ are located on the imaginary axes $\Re(s)=0$, corresponding to the zeros of $g(q^{-s})$, and $\Re(s)=1/4$, corresponding to the zeros of $\zeta_F(2s)$ by the theorem of Hasse-Weil, also known as the function field analogue of the Riemann hypothesis.\footnote{E.g. see Theorem 5.2.1 and Remark 5.2.2, page 197 in \cite{StiALF}.}  
\end{proof}

Note that we could compute the function $\Upsilon_r(s)$ with formula~(\ref{Upsilon}) and the knowledge of only finitely many values of $c_{\mf{m}}$, namely those for $\mf{m}\in\mc{M}$. Hence, the infinite set of modules, which values of $c_{\mf{m}}$ do not behave uniformly, yield a quite manageable error term, but it depends strongly on the given field $F$. The dependence on $G=\cyclic{p}^r$ is very loosely. For the Euler products $\Phi_i(s)$ however, there is virtually no dependence on $F$ and $G$, only it is defined over the primes of $F$, and $i\leq r$ holds, of course. We hence can investigate the Euler product $\Phi_r(s)$ in its own right for any integer $r\geq 1$, and aim to describe $\Phi_r(s)$ as a product of zeta functions.

\begin{prop}\label{prop4-2} The Euler factors $\Phi_{r,\mf{p}}(s)$ of $\Phi_r(s)$ are meromorphic continuable for all complex arguments~$s$ with
$$ \Phi_{r,\mf{p}}(s) = \left(1-\N{p}^{-s}\N{p}^{(p-1)(r-s)}\right)^{-1} \left(1 + \N{p}^{-s}\sum\limits_{l=0}^{p-2} \N{p}^{l(r-s)}\right) \left(1-\N{p}^{-s}\right) .  $$
\end{prop}
\begin{proof} Consider the sum running over $p\nmid n$ of the Euler factor $\Phi_{r,\mf{p}}(s)$ as given in (\ref{Phi}). The occurring $\N{p}$-exponents are of the form $rr_n - (n+1)s$. By definition in Lemma \ref{prop2}, we have $r_n=n-1-\lfloor (n-1)/p\rfloor$. As $n$ is indivisible by $p$, we can write $n=kp+l+1$ for some $k\geq 0$ and $0\leq l\leq p-2$, and $r_n = k(p-1)+l$ follows. Altogether, we obtain
$$ rr_n - (n+1)s = k ((p-1)r-ps) + l(r-s) - 2s  $$
and
$$ \sum\limits_{p\, \nmid\, n} \N{p}^{rr_n-(n+1)s} = \sum\limits_{k\geq 0} \sum\limits_{l=0}^{p-2} \N{p}^{k((p-1)r-ps)}  \N{p}^{l(r-s)} \N{p}^{-2s}. $$
This series is convergent for $\Re(s)>(p-1)r/p$, and meromorphic continuable beyond its abscissa via the rational function
$$ \left(1 - \N{p}^{(p-1)r-ps}\right)^{-1} \ \sum\limits_{l=0}^{p-2} \N{p}^{l(r-s)} \N{p}^{-2s}. $$
The same holds for the Euler factor $\Phi_{r,\mf{p}}(s)$, and its meromorphic continuation has the form
$$ \Phi_{r,\mf{p}}(s) = \left(1 - \N{p}^{(p-1)r-ps}\right)^{-1}  \left(1 - \N{p}^{(p-1)r-ps} + (\N{p}^r - 1) \sum\limits_{l=0}^{p-2} \N{p}^{l(r-s)} \N{p}^{-2s}\right). $$
We hence have verified one proposed factor of $\Phi_{r,\mf{p}}(s)$. Further, the latest term of the right hand side has the identity
\begin{align*} (\N{p}^r - 1) \sum\limits_{l=0}^{p-2} \N{p}^{l(r-s)} \N{p}^{-2s} & = \sum\limits_{l=1}^{p-1} \N{p}^{lr-(l+1)s} - \sum\limits_{l=0}^{p-2} \N{p}^{lr-(l+2)s} \\ & = \N{p}^{(p-1)r-ps} - \N{p}^{-2s} +   \sum\limits_{l=1}^{p-2} \N{p}^{lr-(l+1)s} (1-\N{p}^{-s}).
\end{align*}
Now, extend the latest sum with the index $l=0$ and subtract the corresponding term $\N{p}^{-s} - \N{p}^{-2s}$ to obtain the factorisation of $\Phi_{r,\mf{p}}(s)$ as stated.
\end{proof}

\begin{lemma}\label{prop4-3} The Dirichlet series $\Phi_r(s)$ has the meromorphic continuation $ \Phi_r(s) = \Psi_r(s) \Lambda_r(s) $
for $\Re(s) > \condabsc{p}{r} - \varepsilon$ and some $\varepsilon > 1/(2p)$ with meromorphic factor
$$ \Lambda_r(s) = \prod\limits_{l=2}^p \zeta_F(ls-(l-1)r) $$
and holomorphic factor
$$ \Psi_r(s) = \prod\limits_{\mf{p}} \left( 1 + \N{p}^{-s} \sum\limits_{l=0}^{p-2} \N{p}^{l(r-s)} \right) \prod\limits_{l=0}^{p-2} \left( 1 - \N{p}^{-s}\N{p}^{l(r-s)}  \right). $$
\end{lemma}
\begin{proof} Let $\mf{p}$ be an arbitrary prime of $F$. Then the Euler factor $\Psi_{r,\mf{p}}(s)$ is an entire function, and by Proposition \ref{prop4-2}, we yield the identity of meromorphic functions
$$ \Phi_{r,\mf{p}}(s) \Psi_{r,\mf{p}}(s)^{-1} = \left( 1 - \N{p}^{(p-1)(r-s)-s} \right)^{-1} \prod\limits_{l=1}^{p-2} \left( 1 - \N{p}^{l(r-s)-s}\right)^{-1} = \Lambda_{r,\mf{p}}(s). $$
Observe the substitution $l+1$ for $l$ to get the latter equality. Hence, the Euler factors of $\Phi_r(s)$ and $\Lambda_r(s)\Psi_r(s)$ coincide for all complex arguments $s$. Thus a meromorphic continuation of $\Phi_r(s)$ is given by $\Lambda_r(s)\Psi_r(s)$ for at least in the region of convergence of $\Psi_r(s)$, which is left to be determined. Note the fact on infinite products, that $\prod (1+a_n)$ is absolutely convergent if and only if $\sum a_n$ is absolutely convergent. Furthermore, $\sum \N{p}^{-s}$ is absolutely convergent for $\Re(s)>1$ only. Since the Euler factors of $\Psi_r(s)$ are finite sums of length not depending on the corresponding prime $\mf{p}$, it is sufficient to show, that all $\N{p}$-exponents of $\Psi_{r,\mf{p}}(s)$ are less than $-1$ for real arguments $s>\condabsc{p}{r}-\varepsilon$ and some $\varepsilon>1/(2p)$. There are integers $a(l_1,\ldots,l_k)$, not depending on $\mf{p}$ and zero for almost all $k\geq 1$, such that
$$ \Psi_{r,\mf{p}}(s) = 1 + \sum\limits_{k\geq 1} \  \N{p}^{-ks} \negthickspace \sum\limits_{0\leq l_1, \ldots, l_k \leq p-2} a(l_1,\ldots,l_k) \, \N{p}^{(l_1+\ldots+l_k)(r-s)}.  $$ 
It is easy to see that $a(l)=0$ holds for $0\leq l\leq p-2$. Furthermore, $a(l_1,\ldots,l_k)=0$ holds for $k>p$. Hence, the $\N{p}$-exponents of interest have the form $ -ks + (l_1+\ldots+l_k)(r-s)$ with $2\leq k\leq p$ and $0\leq l_1,\ldots,l_k\leq p-2$. They are less than $-1$ for $s>(1+(l_1+\ldots+l_k)r)/(k+l_1+\ldots+l_k)$. By differentiation for $l_i$, we find these fractions to be maximal for $l_i$ being maximal, and the maximal fraction is given by $(1+k(p-2)r)/(k(p-1))$. In turn, this fraction is maximal for $k$ being minimal. Hence, the $\N{p}$-exponents of $\Psi_{r,\mf{p}}(s)$ are less than $-1$ at least for $ s > (1 + 2(p-2)r)/(2(p-1))$. The distance of this abscissa to $\condabsc{p}{r}$ is given by
$$ \condabsc{p}{r} - \frac{1+2(p-2)r}{2(p-1)} = \frac{1+(p-1)r}{p} - \frac{1+2(p-2)r}{2(p-1)}= \frac{p + 2(r-1)}{2p(p-1)}. $$ 
Hence, we can choose $\varepsilon$ in the range 
$$ \frac{1}{2p} <\varepsilon< \frac{1}{2(p-1)} + \frac{r-1}{p(p-1)} = \frac{p + 2(r-1)}{2p(p-1)}, $$
as desired.
\end{proof}

\begin{cor}\label{prop4-5} The function $\Psi_r(s)$ has the value
$$ \Psi_r(\condabsc{p}{r}) = \prod\limits_{\mf{p}} \left( 1 + \N{p}^{-r} \sum\limits_{l=1}^{p-1} \N{p}^{l(r-1)/p} \right) \prod\limits_{l=1}^{p-1} \left( 1 - \N{p}^{-r}\N{p}^{l(r-1)/p}  \right). $$
\end{cor}
\begin{proof} For $s=\condabsc{p}{r}$, the term $\N{p}^{-s}\N{p}^{l(r-s)}$ of the Euler factor $\Psi_{r,\mf{p}}(s)$ is a $\N{p}$-power with exponent
$$
 -  \frac{1+(p-1)r}{p} + l \left(  r - \frac{1+(p-1)r}{p} \right)  = \frac{r-1-pr}{p}  + \frac{lpr - l - l(p-1)r}{p} = -r + \frac{(l+1)(r-1)}{p}.
$$
This leads to the proposed value of $\Psi_r(s)$ by the substitution $l+1$ for $l$.  
\end{proof}

\begin{cor}\label{prop4-4} For $p=2$, we have
$$ \Phi_r(s) = \zeta_F(2s-r) \, \zeta_F(2s)^{-1}. $$
\end{cor}
\begin{proof} We easily verify $\Lambda_r(s)=\zeta_F(2s-r)$ and $\Psi_r(s)=\zeta_F(2s)^{-1}$ by definition.
\end{proof}

\begin{prop}\label{prop4-6} The function $\Lambda_r(s)$ has the following properties.
\begin{itemize}
 \item[\textup{(a)}] Let $r=1$. Then $\condabsc{p}{}=1$, and $\Lambda_1(s)$ is holomorphic for $\Re(s)>1-1/p$ except for the set of poles
$$  \left\{ 1 + \frac{2 \pi i}{\log(q)}\, \frac{j}{\ell} : 0 \leq j < \ell \right\} + \frac{2\pi i}{\log(q)}\, \ZZ , $$ 
where $\ell$ is the least common multiple of $2,\ldots,p$. The periodic poles corresponding to $s=1$ is of maximal order $\condb{F,\cyclic{p}}=p-1$. In case $p\neq 2$, every other pole has less order. 
 \item[\textup{(b)}] Let $r>1$. Then $\Lambda_r(s)$ is holomorphic for $\Re(s)>\condabsc{p}{r}-\min\{1/p,(r-1)/(p^2-p)\} $ except for the set of poles
$$  \left\{ \condabsc{p}{r} + \frac{2 \pi i}{\log(q)}\, \frac{j}{p} : 0 \leq j < p \right\} + \frac{2\pi i}{\log(q)}\, \ZZ . $$
Every pole has order $\condb{F,\cyclic{p}^r}=1$.
\end{itemize} 
\end{prop}

\begin{proof} Since the poles of $\zeta_F(s)$ are given by the values of $s$ with $q^s = 1$ and $q^s=q$ respectively, we see, that the poles of $\Lambda_r(s)$ are given by the values of $s$ with $q^{ls} = q^{(l-1)r}$ and  $q^{ls} = q^{1+(l-1)r}$for $l=2,\ldots,p$ respectively. Each of these poles has period $2 \pi i/(\log(q)\, l)$. We can ignore the former poles, as they are out of the region of interest due to
$$ \condabsc{p}{r} - \frac{(l-1)r}{l} = \frac{l + (p-l)r}{pl} \geq \frac{1}{p}. $$
For the latter poles, we obtain
$$ \condabsc{p}{r} - \frac{1+(l-1)r}{l} = \frac{(p-l)(r-1)}{lp}. $$ 
This distance is minimal for $l$ being maximal, whence assertion~(b) follows. For assertion~(a), we have $r=1$, and every pole of $\Lambda_1(s)$ in the region $\Re(s)>1-1/p$ lies on the axis $\Re(s)=1$. Since the poles of $\zeta_F(s)$ are simple, the order of any pole of $\Lambda_1(s)$ cannot exceed $p-1$. Since $2\pi i/\log(q)$ is the meet of all periods $2\pi i/(\log(q)\, l)$, only the poles $s\in 1+2 \pi i/\log(q)\, \ZZ$ have maximal order $p-1$, if $p\neq 2$.  
\end{proof}

Now, Theorem~\ref{THEOREMdirichlet} is an easy consequence of the results above. 

\begin{proof}[\textbf{Proof of Theorem~\ref{THEOREMdirichlet}}] By Proposition~\ref{prop4-6}, we only have to show, that $\Psi(s) = \Phi(F,\cyclic{p}^r;s) \Lambda_r(s)^{-1}$ is convergent for $\Re(s)> a$ with $a=\condabsc{p}{r}-1/(2p)$. We have the formal equality
$$ \Psi(s) = e_r \Psi_r(s) + \sum\limits_{i=1}^{r-1} e_i\Phi_i(s)\Lambda_r(s)^{-1} + \Upsilon_r(s)\Lambda_r(s)^{-1} $$
in the region of convergence, and we shall see that the right hand side has convergence abscissa $a$. By Lemma~\ref{prop4-3}, $\Psi_r(s)$ is convergent for $\Re(s)>a$, and so is $\Upsilon_r(s)$ by Lemma~\ref{prop4-1}. Further, $\Phi_i(s)$ has convergence abscissa $\condabsc{p}{i}$, which do not exceed $a$ for $i<r$ due to 
$$ \condabsc{p}{r} - \condabsc{p}{i} = \frac{(p-1)(r-i)}{p} > \frac{1}{2p}. $$  
The zeros of $\Lambda_r(s)$ are located of the imaginary axes $\Re(s)=(1+2(l-1)r)/2l$ for $l=2,\ldots,p$, by the theorem of Hasse-Weil. These fractions are maximal for $l$ being maximal, whence $\Lambda_r(s)^{-1}$ has convergence abscissa $(1+2(p-1)r)/2p=a$. Therefore, $\Psi(s)$ has abscissa $a$, too.
\end{proof}

\section{Distribution of conductors of Artin-Schreier extensions}

This section deals with the proofs of Theorem~\ref{THEOREMasymptoticCond}, Addendum~\ref{THEOREMp2}, and Addendum~\ref{THEOREMr1}. We consider the Dirichlet series
$$ \Phi(F,\cyclic{p}^r;s) = \sum\limits_{\Gal{E}{F} \iso \cyclic{p}^r} \Norm{\mf{f}(E/F)}^{-s} = \sum\limits_{n\geq 0} c_nt^n $$ 
as power series in $t=q^{-s}$ with coefficients $c_n$. Then the counting function $C(F,\cyclic{p}^r;X)$ coincides with coefficients sum of this series via
$$ C(F,\cyclic{p}^r;q^m) = \sum\limits_{n=0}^m c_n.  $$

\begin{proof}[\textbf{Proof of Theorem~\ref{THEOREMasymptoticCond}}] 
By Theorem~\ref{THEOREMdirichlet}, the power series expansion of $\Phi(F,\cyclic{p}^r;s)$ in $t=q^{-s}$ has radius of convergence $R=q^{-\condabsc{p}{r}}$, and it is meromorphic continuable beyond its circle of convergence. By Proposition~\ref{prop4-6}, the poles of the continuation are of the form $t=R\xi^{-j}$ for some root of unity $\xi$. We now obtain Theorem~\ref{THEOREMasymptoticCond} from Theorem~\ref{A5}. The constant $\condc{F,\cyclic{p}^r}$ is hereby a positive real number, since the coefficients $c_n$ are nonnegative integers. 
\end{proof}

Formulas for $\condc{F,\cyclic{p}^r}$ can be obtained by Theorem~\ref{A5}. If we have more than one pole on the circle of convergence with maximal order, we have to use assertion~(b), which involves some tedious calculations. We shall compute $\condc{F,\cyclic{p}^r}$ for $p=2$ or $r=1$ only, as we have only one or two poles of interest in these cases. It is clear, that $C(F,\cyclic{p}^r;X)$ only depends on the coefficient sum of $e_r\Phi_r(s)$. Let $f(t)=e_r\Phi_r(s)$, $g(t)=e_r\Psi_r(s)$, $h(t)=\Lambda_r(s)$, and $Z_F(t)=\zeta_F(s)$ be the corresponding power series expansions of the respective Dirichlet series. Further, let $L_F(t)=(1-t)(1-qt)Z_F(t)$ be the $L$-polynomial of $F$.
Then we have $f(t)=g(t)h(t)$ and
$$ h(t) = \prod\limits_{l=2}^p \frac{L_F(q^{(l-1)r}t^l)}{(1-q^{(l-1)r}t^l)\, (1-q^{1+(l-1)r}t^l)}. $$
Before we start with the calculations of $\condc{F,\cyclic{p}^r}$, the following identities are provided.  

\begin{prop} \label{prop7-1}
\begin{itemize}
\item[\textup{(a)}]  We have
$$ e_r = \prod\limits_{i=0}^{r-1} \frac{p}{p^r-p^i} = \frac{|\cyclic{p}^r|}{|\Aut{\cyclic{p}^r}|}. $$ 
\item[\textup{(b)}]  We have
$$ \frac{L_F(q^{-1})}{1-q^{-1}} = \log(q)\, \zeta_F(1). $$ 
\end{itemize}

\end{prop}
\begin{proof} The first equation in assertion~(a) is just the definition of $e_r$ as given in Proposition~\ref{prop3-1}, and the second one is easy to verify. Assertion (b) follows from $\lim\nolimits_{s\rightarrow 1} (s-1)/(1-q^{1-s}) = 1/\log(q)$. 
\end{proof}

\begin{proof}[\textbf{Proof of Addendum~\ref{THEOREMr1}}] Let $p\neq 2$ and $r=1$. By Proposition~\ref{prop4-6}, $f(t)=g(t)h(t)$ has only one pole on the circle of convergence $|t|=R=q^{-1}$ of maximal order $b=p-1$. Let $c=\lim\nolimits_{t\rightarrow q^{-1}}(t-q^{-1})^{p-1}h(t)$ be the limit of $h(t)$ at this pole. Then we obtain 
$$ \condc{F,\cyclic{p}}  = - \frac{q^{p-1}\, g(q^{-1})\, c}{(p-2)!\, (1-q^{-1})\, \log(q^{-1})^{p-2}} = \frac{q^{p-1}\, g(q^{-1})\, c}{(p-2)!\, (1-q^{-1})\, \log(q)^{p-2}} $$
by Theorem~\ref{A5} (c). Observe that the sign of this constant is positive, since $p$ is an odd prime. Using Corollary~\ref{prop4-5} and Proposition~\ref{prop7-1} (a) helps to identify the corresponding factors of $g(q^{-1})$ in the proposed formula. Now, we only have to verify
$$ c = \lim\limits_{t\rightarrow q^{-1}} (t-q^{-1})^{p-1}\, h(t) = \left(\frac{\log(q)\, \zeta_F(1)}{q}\right)^{p-1}\, \frac{1}{p!}.  $$ 
We obtain the factor $(\log(q)\, \zeta_F(1))^{p-1}$ from the limit of $\prod\nolimits_{l=2}^p Z_F(q^{l-1}t^l)/(1-q^{l-1}t^l)$ for $t\rightarrow q^{-1}$ via assertion~(b) of Proposition~\ref{prop7-1}. Finally, we have
$$ \lim\limits_{t\rightarrow q^{-1}} (t-q^{-1})^{p-1}\, \prod\limits_{l=2}^p \frac{1}{1-(qt)^l} = \lim\limits_{t\rightarrow q^{-1}} \left(\frac{qt-1}{1-qt} \right)^{p-1} \prod\limits_{l=2}^p \frac{1}{q\, \sum\nolimits_{k=0}^{l-1} (qt)^k} = \frac{1}{q^{p-1}\, p!}. $$ 
Putting all together, we obtain the formula for $\condc{F,\cyclic{p}}$ as stated.
\end{proof}

\begin{proof}[\textbf{Proof of Addendum~\ref{THEOREMp2}}] Let $p=2$ and $r\geq 1$. Then $f(t)=e_r Z_F(q^rt^2)/Z_F(t^2)$ by Corollary~\ref{prop4-4}. We have two simple poles on the circle of convergence $|t|=R=q^{-(1+r)/2}$, namely $t=\pm R$. A priori, we have two relevant asymptotic equivalence classes for $C(F,\cyclic{2}^r;X)$ by Theorem~\ref{A5} (b), namely one for $X=q^{m}$ with $m$ running over even integers, and the one for $m$ running over odd integers. Note that we only have to prove, that one of these equivalence classes coincides with the proposed one, by the redefinition of asymptotic equivalence as declared at the beginning of the article. By Theorem \ref{A5}, we have
$$ \gamma_e = \lim\limits_{n\rightarrow\infty} \frac{C(F,\cyclic{2}^r;q^{2n+e})}{R^{-(2n+e)}} = - \left( \frac{R^{-1}\, \Res_{t=R}(f(t))}{1-R} + \frac{(-R)^{-1}\, \Res_{t=-R}(f(t))}{1+R}\, (-1)^e \right) $$ 
for $e=0,1$. The residues for $t=\pm R$ only differ in their sign, and we obtain
$$ \Res_{t=\pm R}(f(t)) =  \frac{e_r\, \log(q)\, \zeta_F(1)}{\zeta_F(r+1)\, 2(\pm R)^{-1}}. $$
Now, the constant $\gamma_e$ differs from the proposed formula of $\condc{F,\cyclic{2}^r}$ by the factor $((1+R) + (-1)^e (1-R))/2$, and we obtain 
$\condc{F,\cyclic{2}^{r}} = \gamma_0 = \gamma_1/R$. Hence, the proposed equivalence class for $C(F,\cyclic{2}^r;X)$ is true for $X=q^{2n}$ as $n\rightarrow\infty$. 
\end{proof}

\begin{appendix}
\section{Tauberian theorem for power series} 

In this section, we seek accurate asymptotic formulas for the coefficients and coefficient sums of a power series $f(t) = \sum\nolimits_{n\geq 0} c_nt^n$. We will follow the lead of Michael Rosen's proof for Theorem 17.4 on page 311 in \cite{RosenNTFF}, which shall be generalised with part (a) of Lemma~\ref{A1} and Theorem~\ref{A5} for the sake of completeness. For the purpose of the article at hand, we are more interested in the coefficient sums, especially the parts (b) and (c) in latter-mentioned theorem. As a general assumption for this appendix, let $R$ be the radius of convergence with $0<R<1$, and let $f(t)$ be holomorphic continuable in $|t|\leq R + \varepsilon$ for some $\varepsilon>0$ with exception for a discrete pole set $P$, which is contained in the circle of convergence $|t|=R$. Further, let $p_u((t-u)^{-1}) = \sum\nolimits_{k=1}^b p_{u,k}(t-u)^{-k}$ be the principal part of the Laurent series of $f(t)$ for $t=u$, and let $p_{u,b}\neq 0$ for at least one pole $u\in P$.

\begin{prop}\label{A2} Let $g_n(t)=1/t^{n+1}$ and $h_m(t) = \sum\nolimits_{n=0}^m g_n(t)$, and let $g_n^{(l)}(t)$, $h_m^{(l)}(t)$ denote their respective $l$-th derivative. Then the following holds.
\begin{itemize}
 \item[\textup{(a)}] For any integer $n\geq 0$, we have 
$$ \frac{g_n^{(l)}(t)}{l!} = - (-t)^{-(l+1)} \binom{n+l}{l}  t^{-n}. $$
 \item[\textup{(b)}] For any integer $m\geq 0$,  we have
$$ \frac{h_m^{(l)}(t)}{l!} =  - (-t)^{-(l+1)}  \sum\limits_{n=0}^m \binom{n+l}{l}  t^{-n}. $$   
\end{itemize}
\end{prop}
\begin{proof} Assertion (a) follows from $ g_n^{(l)}(t) = (-1)^lt^{-(n+l+1)} \prod\nolimits_{i=1}^l (n+i). $
By adding from $n=0$ up to $n=m$, we hence obtain assertion~(b).
\end{proof}

\begin{lemma}\label{A1} Let $f(t)=\sum\nolimits_{n\geq 0} c_nt^n$ be a power series with properties as given above.
\begin{itemize}
 \item[\textup{(a)}] The coefficient $c_n$ fulfils the asymptotic relation
$$ \left|\, c_n \, - \, \sum\limits_{u\in P} \sum\limits_{k=1}^b (-u)^{-k}\, p_{u,k} \, \binom{n+k-1}{k-1}\, u^{-n} \, \right| \in \kleino{R^{-n}}. $$
 \item[\textup{(b)}] The coefficient sum $\sum\nolimits_{n=0}^m c_n$ fulfils the asymptotic relation
$$ \left|\, \sum\limits_{n=0}^m c_n \, - \, \sum\limits_{u\in P}  \sum\limits_{k=1}^b (-u)^{-k}\, p_{u,k} \, \sum\limits_{n=0}^m \binom{n+k-1}{k-1}\, u^{-n} \, \right| \in \kleino{R^{-m}}. $$
\end{itemize}
\end{lemma}
\begin{proof} For $m\geq 0$, let $e_m(t)=g_m(t)$ or $e_m(t)=h_m(t)$. Let $\gamma$ be an anticlockwise oriented circle path with winding number~$1$ and radius~$R+\varepsilon$ for some $0<\varepsilon<1-R$ . By the residue theorem, we have 
$$  \Res_{t=0}(f(t)e_m(t)) + \sum\limits_{u\in P} \Res_{t=u}(f(t)e_m(t)) = \frac{1}{2\pi i}\ \Xint\circlearrowleft_{\gamma} f(t)e_m(t)\, dt.  $$
The right hand side including the integral can be bounded by $(R+\varepsilon)\max\nolimits_{\gamma}(f(t)) \max\nolimits_{\gamma}(e_m(t))$. The former maximum of this product exists, since $\gamma$ is compact, and the latter maximum is of the form $\kleino{R^{-(m+1)}}$, since $R<R+\varepsilon<1$. In particular, $h_m(t)=t^{-(m+1)}(1-t^{m+1})/(1-t)$ holds. Hence, we have
\begin{equation}\label{FormelA2} \left|\, \Res_{t=0}(f(t)e_m(t)) + \sum\limits_{u\in P} \Res_{t=u}(f(t)e_m(t)) \, \right| \in \kleino{R^{-m}}. 
\end{equation}
The residue for $t=0$ is $c_m$ in case of $e_m(t)=g_m(t)$, and $\sum\nolimits_{n=0}^m c_n$ in case of $e_m(t)=h_m(t)$. Further, the residues for the poles $t=u\in P$ are of the form
$$ \Res_{t=u}(f(t)e_m(t)) = \sum\limits_{k-l=1} p_{u,k}\, \frac{e_m^{(l)}(u)}{l!} = \sum\limits_{k=1}^b p_{u,k}\, \frac{e_m^{(k-1)}(u)}{(k-1)!}, $$
where the summands are coefficient products of the Laurent series for $f(t)$ and the Taylor series for $e_m(t)$ in $t=u$ respectively. 
This leads to the proposed relations by formula (\ref{FormelA2}) and Proposition~\ref{A2}.
\end{proof}

\begin{prop}\label{A4} We have the asymptotic relation
 $$ \left|\, \sum\limits_{n=0}^m \binom{n+l}{l}\, t^{-n} \, - \, \frac{1}{l!\, (1-t)}\, t^{-m}\, m^l \, \right| \in\kleino{R^{-m}\, m^l}   $$
 for $t\in\CComplex$ being fixed with $|t|=R$, and $m\rightarrow\infty$. 
\end{prop}
\begin{proof} We have the algebraic relation
 $$ (1-T)^{l+1} \sum\limits_{n=0}^m \binom{n+l}{l}\, T^n = 1 - \sum\limits_{i=0}^l \binom{m+i}{i}\, (1-T)^{i}\, T^{m+1}, $$ 
a formula for the first $m$ terms of the binomial series expansion for $(1-T)^{-(l+1)}$, which follows from induction over $l$.   
If we substitute $T$ for a complex number $z$ of absolute value $|z|=1/R>1$, we have
$$ \sum\limits_{n=0}^m \binom{n+l}{l}\, z^n = \frac{1}{(1-z)^{l+1}} - \sum\limits_{i=0}^{l-1} \binom{m+i}{i}\, \frac{z}{(1-z)^{l+1-i}}\, z^m - \binom{m+l}{l}\, \frac{z}{1-z}\, z^m. $$
The binomial coefficient $\binom{n+i}{i}$ is a polynomial in $n$ of degree $i$ with leading coefficient $1/i!$. By subtracting the leading term of the right hand side, we obtain
$$ \left|\, \sum\limits_{n=0}^m \binom{n+l}{l}\, z^n \, + \, \frac{m^l}{l!}\, \frac{z}{1-z}\, z^m \,\right| \leq c\, R^{-m}\, m^{l-1}, $$
where $c$ is a constant not depending on $m$. Now, the proposed bound follows with the choice $z=1/t$ and $|t|=R<1$. 
\end{proof}

\begin{lemma}\label{A3} Let $f(t)=\sum\nolimits_{n\geq 0} c_nt^n$ be a power series with properties as given above. Further, assume the pole set $P$ to be of the form $P=\{ R\xi^{-j} : 1\leq j\leq \ell\}$ for a primitive $\ell$-th root of unity $\xi\in\CComplex$, and let $p_{j}=p_{u,b}$ for $u=R\xi^{-j}$. Then the following holds. 
\begin{itemize}
 \item[\textup{(a)}] The coefficient $c_n$ fulfils the asymptotic relation 
$$ \left|\, c_n \, - \, \left( \sum\limits_{j=1}^{\ell} \frac{(-R\xi^{-j})^{-b}\, p_{j}}{(b-1)!}\, \xi^{jn} \right) R^{-n}\, n^{b-1} \, \right| \in \kleino{R^{-n}\, n^{b-1}}. $$
For $n\rightarrow\infty$ running over some arithmetic progression modulo $\ell$, we obtain $c_n\in\grossTheta{R^{-n}\, n^{b-1}}$.
 \item[\textup{(b)}] The coefficient sum $\sum\nolimits_{n=0}^m c_n$ fulfils the asymptotic relation
$$ \left|\, \sum\limits_{n=0}^m c_n \, - \,  \left( \sum\limits_{j=1}^{\ell} \frac{(-R\xi^{-j})^{-b}\, p_{j}}{(b-1)!\, (1-R\xi^{-j})}\, \xi^{jm} \right) R^{-m}\, m^{b-1} \, \right| \in \kleino{R^{-m}\, m^{b-1}}. $$
For $m\rightarrow\infty$ running over some arithmetic progression modulo $\ell$, we obtain $\sum\nolimits_{n=0}^m c_n\in\grossTheta{R^{-m}\, m^{b-1}}$.
\end{itemize}
\end{lemma}

\begin{proof} The formulas are an easy consequence from Lemma~\ref{A1} and Proposition~\ref{A4}. By independence of characters, the inner sum running over the index $j$ is nonzero for at least one congruence class modulo $\ell$. By considering only integers of the so given arithmetic progression, we yield the size of the coefficients and their sums respectively.
\end{proof}

\begin{theorem}\label{A5} Let $f(t)=\sum\nolimits_{n\geq 0} c_nt^n$ with pole set $P$ be given as in Lemma~\ref{A3}. 
\begin{itemize}
 \item[\textup{(a)}] There is some integer $1\leq e\leq \ell$, such that the coefficient $c_n$ has the asymptotic equivalence class
$$ c_n \sim \left( - \sum\limits_{j=1}^{\ell} \frac{(R\xi^{-j})^{-b}\, p_{j}}{(b-1)!\, \log(R)^{b-1}}\, \xi^{je} \right) X\, \log(X)^{b-1} $$
for $X=R^{-n}$ and $n\rightarrow\infty$ running over the arithmetic progression $n\equiv e \pmod{\ell}$.
 \item[\textup{(b)}] There is some integer $1\leq e\leq \ell$, such that the coefficient sum $\sum\nolimits_{n=0}^m c_n$ has the asymptotic equivalence class
$$ \sum\limits_{n=0}^m c_n \sim \left( - \sum\limits_{j=1}^{\ell} \frac{(R\xi^{-j})^{-b}\, p_{j}}{(b-1)!\, (1-R\xi^{-j})\, \log(R)^{b-1}}\, \xi^{je} \right) X\, \log(X)^{b-1}  $$
for $X=R^{-m}$ and $m\rightarrow\infty$ running over the arithmetic progression $m\equiv e \pmod{\ell}$.
 \item[\textup{(c)}] Assume $p_{u,b}\neq 0$ for $u=R$ only. Then we have
$$ \sum\limits_{n=0}^m c_n \sim  \left( -\frac{R^{-b}\, p_{R,b}}{(b-1)!\, (1-R)\, \log(R)^{b-1}}\right) X\, \log(X)^{b-1} $$
for $X=R^{-m}$ and $m\rightarrow\infty$.
\end{itemize}
\end{theorem}
\begin{proof} This theorem is just a restatement of Lemma~\ref{A3} with regard of $n=-\log(R^{-n})/\log(R)$. 
\end{proof}

\textbf{Acknowledgement}. I was funded by Deutsche Forschungsgesellschaft via a grant at Berlin Mathematical School and via the priority project SPP 1489 KL 1424/8-1. My special thanks go to Florian He\ss, J\"urgen Kl\"uners, and Florin Nicolae for fruitful discussions on this work.

\end{appendix}

\end{document}